\documentclass[a4paper,12pt]{amsart}
\usepackage[top=3cm,bottom=3cm,outer=3cm,inner=3cm,marginpar=2.45cm]{geometry}
\usepackage{latexsym,amsmath,mathrsfs}
\usepackage{amsthm,amssymb}

\usepackage[all]{xy}
\usepackage[utf8]{inputenc}
\usepackage{hyperref}
\newtheorem{theorem}{Theorem}[section]

\newtheorem{corollary}[theorem]{Corollary}
\newtheorem{proposition}[theorem]{Proposition}

\def\we{\wedge}

\newcommand{\wt}{\widetilde}
\newcommand{\pt}{\partial}

\def\psh{plurisubharmonic}
\def\ii{{\sqrt{-1}}}

\def\cy{Calabi-Yau}
\def\wp{Weil-Pe\-ters\-son}
\def\psh{plurisubharmonic}
\def\ka{K\"ah\-ler}
\def\ma{Monge-Amp\`ere}
\def\ke{K\"ah\-ler-Ein\-stein}
\def\he{Her\-mite-Ein\-stein}

\numberwithin{equation}{section}

\def\RR{\mathbb{R}} % the real number R
\def\CC{\mathbb{C}} % Complex C
\def\NN{\mathbb{N}} % Natural number N
\def\im{\sqrt{-1}} % the imaginary number i
\def\ddbar{\partial\bar\partial} %\ve
\def\vol{\mathrm{Vol}} %volume

\def\vp{\varphi}
\def\beq{\begin{equation*}}
\def\eeq{\end{equation*}}

\def\wp{Weil-Pe\-ters\-son }

\newcommand{\paren}[1]{\left(#1\right)}

\newcommand{\abs}[1]{\left\vert#1\right\vert}
\newcommand{\norm}[1]{\left\|#1\right\|}

\newcommand{\ip}[1]{\mathrm{Im}\;s}

\begin{document}

\title[Curvature of the relative canonical line bundle]{Extension of the curvature form of the relative canonical line bundle on families of Calabi-Yau manifolds and applications}

\author{Young-Jun Choi}

\address{Department of Mathematics, Pusan National University, 2, Busandaehak-ro \break 63beon-gil, Geumjeong-gu, Busan, 46241, Korea}

\email{youngjun.choi@pusan.ac.kr}

\author{Georg Schumacher}

\address{Fachbereich Mathematik und Informatik, Philipps-Universit\"at Marburg, Lahnberge, Hans-Meerwein-Stra{\ss}e, D-35032 Marburg, Germany}

\email{schumac@mathematik.uni-marburg.de}



\begin{abstract}
Given a proper, open, holomorphic map of \ka\ manifolds, whose general fibers are \cy\ manifolds, the volume forms for the Ricci-flat metrics induce a hermitian metric on the relative canonical bundle over the regular locus of the family. We show that the curvature form extends as a closed positive current. Consequently the \wp\ metric extends as a positive current. In the projective case, the \wp\ form is known to be the curvature of a certain determinant line bundle, equipped with a Quillen metric. As an application we get that after blowing up the singular locus, the determinant line bundle extends, and the Quillen metric extends as singular hermitian metric, whose curvature is a positive current.
\end{abstract}

\maketitle

\section{Introduction}\label{S:int}
When studying moduli spaces, the extension of the \wp\ metric into the boundary of the moduli space, more precisely in the situation of a degenerating family, is of interest. This could be done for families of canonically polarized manifolds or stable holomorphic vector bundles. The approach in both cases is to use a fiber integral formula for the \wp\ metric in terms of a certain positive $(n+1,n+1)$-form on the total space (where $n$ denotes the dimension of the fibers). Such a form has to be extended as a positive current, and a push forward is taken. This approach is closely related to the existence of \ke\ and \he\ metrics resp. In the former case, uniform $C^0$-estimates for the solution of the resp.\ \ma\ equations in a degenerating family ultimately provide the extension property.

Here, we will treat the case of \cy\ manifolds, i.e.\ polarized \ka\ manifolds with vanishing first real Chern class, equipped with Ricci-flat metrics according to Yau's theorem \cite{Yau}. Our result will be based upon the Ohsawa-Takegoshi extension theorem for canonical forms on submanifolds in bounded pseudoconvex domains. The statement of the main theorem will not depend on the type of the singularities of the degeneration.

Let $f: X\rightarrow S$ be a proper surjective holomorphic mapping between complex manifolds, where $X$ is \ka. We denote by $W\subset S$ the analytic set of singular values of $f$, and let $X':=f^{-1}(S')$ with $S'=S\setminus W$. Suppose that for every $s\in S'$, the fiber $X_s=f^{-1}(s)$ is a Calabi-Yau manifold, i.e.\ a compact \ka\ manifold, whose first real Chern class $c_1(X_s)$ vanishes.
Let $\omega$ be a \ka\ form on $X$. Then Yau's theorem implies that there exists a unique Ricci-flat metric $\omega^{KE}_s$ in the class of $\omega\vert_{X_s}$ on each smooth fiber $X_s$, $s\in S'$. One can see that the family of Ricci-flat metrics induces a fiberwise Ricci-flat $d$-closed smooth real $(1,1)$-form $\rho$ on $X'$ such that
\begin{equation*}
\rho\vert_{X_s} = \omega^{KE}_s.
\end{equation*}
It can be normalized by the condition
\begin{equation}\label{E:star}
f_*(\rho^{n+1}) = \int_{X'/S'} \rho^{n+1} =0, \text{ where } n=\dim X_s
\end{equation}
in which case $\rho$ is unique (cf.\ \cite{Braun:Choi:Schumacher}). By adding the pull-back of a \ka\ form on $S'$ to $\rho$, (like the fiber integral $\int_{X'/S'} \omega^{n+1}$, or  the \wp form, if the family is effectively parameterized) we can achieve that $f_*(\rho^{n+1})$ is positive definite. Conversely, if $f_*(\rho^{n+1})>0$, then $\rho - f^*(f_*(\rho)/\vol(X_s))$ satisfies the former normalization condition. A form $\rho$ satisfying the latter condition is also called {\em fiberwise Ricci-flat metric}. A fiberwise Ricci-flat metric $\rho$ is not uniquely determined, but it always exists globally on $X'$ under our assumptions. The corresponding {\em relative volume form} on $X' \to S'$ is independent of the normalization, however, later we will use \eqref{E:star}. It
induces a hermitian metric $h^\rho_{X'/S'}$ on the relative canonical line bundle (see Section \ref{S:Preliminaries}). Now the Weil-Petersson metric $\omega^{WP}_{S'}$ and the curvature $\Theta_{h^\rho_{X'/S'}}(K_{X'/S'})$ of  $h^\rho_{X'/S'}$ defined on $S'$ satisfy the following equation according to \cite{Braun:Choi:Schumacher}. (Note that the volume of the smooth fibers is constant.)
\begin{equation}\label{E:wp}
\frac{1}{\vol(X_s)}f^*\omega^{WP} = \Theta_{h^\rho_{X'/S'}}(K_{X'/S'}).
\end{equation}
This implies that the curvature $\Theta_{h^\rho_{X'/S'}}(K_{X'/S'})$ of the relative canonical line bundle is semi-positive on $X'$. The main goal of this paper is the extension of $\Theta_{h^\rho_{X'/S'}}(K_{X'/S'})$.

\begin{theorem}\label{T:main_theorem}
Let $f:X\rightarrow S$ be a proper, open, surjective holomorphic mapping of \ka\ manifolds. Suppose that every regular fiber $X_s$ is a Calabi-Yau manifold for $s\in S'$. Then the curvature $\Theta_{h^\rho_{X'/S'}}(K_{X'/S'})$ of the relative canonical line bundle extends to $X$ as a $d$-closed positive $(1,1)$-current $\Theta$.
\end{theorem}
On $S'$
$$
\omega^{WP}= f_*\left(f^*(\omega^{WP})\we \omega^n\right) = \int_{X'/S'}f^*(\omega^{WP})\we \omega^n,
$$
holds, where $\omega$ can be replaced by any \ka\ form that induces the same polarization on the fibers.

With wedge products taken in the sense of Bedford and Taylor, the Theorem together with \eqref{E:wp} will imply that
$$
\int_{X/S} \Theta \we \omega^n
$$
defines an extension of $\omega^{WP}$ as a positive, real current.

\begin{corollary}\label{C:wp}
The \wp\ form extends to all of $S$ as a positive, closed current.
\end{corollary}
The proof of the Theorem uses is inspired by a result of P\u{a}un from \cite{Paun}, applying Demailly's approximation theorem of \psh\ functions by logarithms of absolute values of holomorphic functions \cite{dem}.%, which for its part relies on the Ohsawa-Takegoshi Theorem.

In \cite[Sect.\ 10]{f-s}, for projective families of Ricci-flat manifolds (more generally extremal \ka\ manifolds), we showed that a certain determinant line bundle $\lambda$ on the base of the family could be equipped with a Quillen metric $h^Q$ such that the curvature form is up to a numerical constant equal to the \wp\ form. Using the extension theorem from \cite{Schumacher,sch:ext}, we can see that after a blow-up $(\lambda,h^Q)$ can be extended into the singular locus of such a projective family.

\section{Preliminaries}\label{S:Preliminaries}
\subsection{Construction of a background metric}\label{S:background}
Let $f:X\to S$ be a proper, surjective holomorphic mapping between K\"ahler manifolds. Let $d=\dim S$, and $\dim X= n+d$. Let $S'\subset S$ be the set of regular values and $X'=f^{-1}(S')$ as above. Then the restriction $f':X' \rightarrow S'$ is a surjective holomorphic submersion.

Let $\omega$ be a \ka\ form on $X$, or more generally a smooth, $d$-closed, real $(1,1)$ form, which is positive definite along the smooth fibers of $f$. We will use the notations $\omega^n= \omega\we\ldots\we\omega/n!$ and analogous.

Since the relative canonical bundle $K_{X/S}$ is defined as $K_X\otimes f^*K^{-1}_S$ a hermitian metric $h^\omega_{X/S}$ can be defined as follows:

Given any local coordinate system $(z^1,\ldots,z^{n+d})$ on $U\subset X$ we denote the corresponding Euclidean volume form by $dV_X=dV_X(z)$, and for local coordinates $(s^1,\ldots,s^d)$ on the image of $U$ under $f$ in $S$  we use the Euclidean volume form $dV_S=dV_S(s)$.
Let
\begin{equation}\label{E:chi}
\omega^n \we f^*dV_S = \chi \cdot dV_X,
\end{equation}
with a certain non-negative, differentiable function $\chi$. At singular points of $f$ the function $\chi$ will vanish, and we write
\begin{equation}\label{E:psi}
\chi = e^{\psi_U},
\end{equation}
where $\psi_U$ takes the value $-\infty$ at singular points of $f$.

On the regular locus of $f$ in $X$ the function $\chi$ is the absolute value squared of a holomorphic function so that $e^{-\psi_U}$ defines a singular hermitian metric  $h^\omega_{X/S}$ on $K_{X/S}$.

The corresponding curvature form on $X'$ is given by
\begin{equation}\label{E:tau}
\ii\Theta_{h^\omega_{X/S}}(K_{X/S}) = dd^c\psi_U|X' = dd^c\log\paren{\omega^n\wedge f^*dV_S}.
\end{equation}
So far, the curvature is a positive current on $X$, and differentiable on the regular locus of $f$ in $X$.

\subsection{Existence and curvature of a fiberwise Ricci-flat metric}\label{S:fiberwiseRFf}

Let the restriction of $f$ now denote a smooth family of \cy\ manifolds $f':X'\to S'$, and let  $\omega$ to be a \ka\ form  on $X$. Let $\omega_s=\omega|_{X_s}$. For $s\in S'$ we take the \ka\ class $[\omega_s]$ as polarization, and apply Yau's theorem.

Since $[\textrm{Ric}(\omega_s)]=0$, for any $s\in S'$, by the $dd^c$-lemma, there exists a unique function $\eta_s\in C^\infty(X_s)$ such that
\begin{eqnarray}
	 dd^c\eta_s &=&\text{Ric}(\omega_s)\label{E:ddceta}\\
	\int_{X_s} e^{\eta_s}\omega_s^n &=&\int_{X_s}\omega_s^n =\vol(X_s)\label{E:norm}.
\end{eqnarray}
Then the implicit function theorem says that the function $\eta$, which is defined by $\eta(x)=\eta_s(x)$ for $s=f(x)$, is a smooth function on $X'$.

There exists a $d$-closed, smooth, real $(1,1)$-form $\rho$ such that $\rho\vert_{X_s}$ is the Ricci-flat metric on the polarized manifold $(X_s,[\omega_s])$. This follows, since $X'$ is assumed to be \ka (cf.\ \cite{Braun:Choi:Schumacher}): Take $\omega_s$ as initial metric for the solution of the \ma\ equation in Yau's solution of the Calabi problem, and set $\rho_s=\omega_s+dd^c \varphi_s$ on $X_s$, $s\in S'$. (Here we use the assumption that $\omega$ is globally defined.) Then the form
$$
\rho= \omega + dd^c \varphi \text{ on } X'
$$
restricts to $\rho_s$ on each fiber $X_s$. It is possible to normalize $\rho$ in the sense of \eqref{E:star} or to make it a fiberwise Ricci-flat metric.

Using \eqref{E:ddceta} and \eqref{E:norm} we obtain
\begin{equation}\label{E:rho}
e^{\eta_s}\omega_s^n=\rho_s^n
\end{equation}
for each $s\in S'$.
Multiplying this equation by the Euclidean volume form $dV_S$, we get
\begin{equation*}
e^\eta\omega^n\wedge f^*(dV_S)=\rho^n\wedge f^*(dV_S)
\end{equation*}
on $X'$.
Hence
\begin{equation*}
dd^c\log\paren{e^\eta\omega^n\wedge f^*(dV_S)} = dd^c\log\paren{\rho^n\wedge f^*(dV_S)}
\end{equation*}
holds over $S'$.
Now we apply \eqref{E:tau} to the left-hand side of the above equation. Concerning the right-hand side, the same argument provides us with the curvature of the relative canonical bundle with respect to $\rho$:
\begin{equation}\label{E:curvature}
dd^c (\eta +\psi_U)=dd^c\eta + \ii\Theta_{h^\omega_{X'/S'}}(K_{X'/S'})=\ii\Theta_{h^\rho_{X'/S'}}(K_{X'/S'}).
\end{equation}

\subsection{Deformation theory of polarized \cy\ manifolds, and the \wp metric}
For general facts about the \wp\ metric, in particular \eqref{E:wp}, we refer to \cite{Braun:Choi:Schumacher}. Let $f:X\to S$ be a (smooth) polarized, holomorphic family of \cy\ manifolds, and $\rho$ be a form as above such that the restriction to a fiber $X_s$ is  the Ricci-flat form on the polarized fiber with the normalization
\begin{equation}\label{eq:normrho}
\int_{X/S} \rho^{n+1}=0.
\end{equation}
Previously, according to \cite[Thm.\ 7.8, Rem.\ 7.3]{f-s}, the following fiber integral formula for the \wp\ metric was known:
\begin{equation}\label{E:wpfib}
\omega^{WP}= \int_{X/S} \Theta_{h^\rho_{K_{X/S}}} \we \omega^n,
\end{equation}
which clearly follows from \eqref{E:wp}. Such fiber integrals fit into the framework of the Riemann-Roch-Hirzebruch-Grothendieck formalism (cf.\ Section~\ref{S:proj}).

\section{Extension of the curvature form of the \\ relative canonical line bundle}
In this section, we will derive an upper bound of a local potential of the curvature $\Theta'=\Theta_{h^\rho_{X'/S'}}(K_{X'/S'})$, which is induced by $\rho$ on the relative canonical line bundle and prove Theorem~\ref{T:main_theorem}. Since this curvature is known to be positive, the boundedness from above of a local potential near singular fibers will yield a local extension of the curvature current $\Theta'$ into singular fibers of the total space. The null-extensions (cf. \cite[(1.19)]{dem})  of these currents will fit together and yield a global positive current.

\medskip

Let $x_0$ be a point which belongs to a singular fiber $X_{s_0}$, i.e., $s_0\in W$.
Let $(U,z^1,\dots,z^{n+d})$ be a local coordinate at $x_0$ in $X$ and $(f(U),s^1,\dots,s^d)$ be a local coordinate at $s_0$ in $S$.
Take a neighborhood $V\subset U$ of $x_0$ and $r>0$ such that $B_r(x)\subset U$ for any $x\in V$, where $B_r(x)$ is a geodesic ball centered at $x$ with radius $r$ with respect to $\omega$.
Let $\psi_U$ be the local potential of $\Theta_{h^\omega_{X/S}}(K_{X/S})$ which is defined in Section \ref{S:background}.
Define a function $\theta:U\cap X'\rightarrow\RR$ by
\begin{equation}\label{E:potential}
\theta := \psi_U+\eta.
\end{equation}
Then $\theta$ is a local potential of $\Theta_{h^\rho_{X'/S'}}(K_{X'/S'})$ in $U\cap X'$ by \eqref{E:curvature}.
Since $\theta$ is plurisubharmonic, it is enough for the extension of $\Theta_{h^\rho_{X'/S'}}(K_{X'/S'})$ to show that $\theta$ is bounded from above in $V\cap X'$.

Fix a point $x\in V\cap X'$ and denote by $s=f(x)$.
Then $X_s$ is a smooth fiber and $U_s:=U\cap X_s$ is a Stein manifold of $U$.
We claim that the restriction of $\theta$ to $U_s$
\begin{equation*}
\theta_s:=\theta\vert_{U_s}
=
\psi_s+\eta_s
\end{equation*}
is uniformly bounded from above in $V_s$.
We now adopt an argument of M.~P\u{a}un from \cite{Paun} to apply an approximation argument of Demailly. The idea is to approximate the function $\theta_s$ by the logarithm of absolute values of holomorphic functions.

\begin{theorem}\label{T:demailly}
Let $\mathcal H_s^{(m)}$ be the Hilbert space defined as follows:
\begin{equation*}
\mathcal H_s^{(m)}
:=
\left\{
\phi\in\mathcal O(U_s):
\norm{\phi}_{m,s}^2
:=
\int_{U_s}
\abs{\phi}^2
e^{-m\theta_s}
(\rho_s)^n
<
\infty
\right\}.
\end{equation*}
Then we have
\begin{equation*}
\theta_s(x)
=
\lim_{m\rightarrow\infty}
\sup\left\{
	\frac{1}{m}\log\abs{\phi(x)}^2:
	\phi\in\mathcal H_s^{(m)}
	\;\;\text{s.t.}\;\;\norm{\phi}_{m,s}^2\le1
\right\}
\end{equation*}
for every $x\in U_s$.
\end{theorem}

\begin{proof}
The proof is essentially same with the proof of Theorem 14.2 in \cite{dem}.
We may assume that $U$ is biholomorphic to a bounded strongly pseudoconvex domain in $C^{n+d}$, so $U_s$ is a Stein manifold.
This implies that $\mathcal{H}_s^{(m)}$ is a nontrivial seperable Hilbert space.
Hence there exists an orthonormal basis $\{\sigma^s_k\}_{k\in\NN}$ for $\mathcal{H}_s^{(m)}$.
Now we define the Bergman kernel $\vp_m^s$ by
\begin{equation*}
\vp_m^s(x)
=
\frac{1}{m}\log\sum_{k=1}^\infty\abs{\sigma_k^s(x)}^2.
\end{equation*}
It follows immediately that the Bergman kernel satisfies the extremal property:
\begin{equation*}
\vp_m^s(x)
=
\sup\left\{
	\frac{1}{m}\log\abs{\phi(x)}^2:
	\phi\in\mathcal H_s^{(m)}
	\;\;\text{s.t.}\;\;\norm{\phi}_{m,s}^2\le1
\right\}.
\end{equation*}

Take a local coordinate neighborhood $B_\varepsilon:=\{\zeta\in U_s:\abs{x-\zeta}<\varepsilon\}$ in $U_s$. Denote by $d\lambda$ the Lebesgue measure in $B_\varepsilon$. For any $\phi\in\mathcal H_s^{(m)}$ with $\norm{\phi}_{m,s}^2\le1$, for $0<r<\varepsilon$, the mean value inequality implies
\begin{align*}
\abs{\phi(x)}^2
&\le
\frac{n!}{\pi^nr^{2n}}
\int_{B_r}\abs{\phi(\zeta)}^2 d\lambda(\zeta)\\
&\le
\frac{n!}{\pi^nr^{2n}}
\sup_{\zeta\in B_r}
\paren{e^{m\theta_s(\zeta)}}
\int_{B_r}\abs{\phi(\zeta)}^2e^{-m\theta_s} d\lambda(\zeta).
\end{align*}
Let $C_s:=\sup_{B_r} d\lambda/(\rho_s)^n$.
Then it follows that
\begin{align*}
\abs{\phi(x)}^2
&\le
\frac{n!C_s}{\pi^nr^{2n}}
\sup_{\zeta\in B_r}
\paren{e^{m\theta_s(\zeta)}}
\int_{B_r}\abs{\phi(\zeta)}^2e^{-m\theta_s}(\rho_s)^n\\
&\le
\frac{n!C_s}{\pi^nr^{2n}}
\sup_{\zeta\in B_r}
\paren{e^{m\theta_s(\zeta)}}.
\end{align*}
Hence we have
\begin{equation*}
\frac{1}{m}\log\abs{\phi(x)}^2
\le
\sup_{\zeta\in B_r}\theta_s(\zeta)
+
\frac{1}{m}\paren{\log\frac{n!}{\pi^nr^{2n}}+\log C_s}.
\end{equation*}
Since $\theta_s$ is upper-semicontinuous, taking $r=1/m$ it follows that
\begin{equation*}
\theta_s(x)\ge\lim_{m\rightarrow\infty}\vp_m^s(x).
\end{equation*}

Conversely, first note that
$$
\sigma=\frac{dz^1\wedge\dots\wedge dz^{n+d}}{f^*(ds^1\wedge\dots\wedge ds^d)}
$$
is a holomorphic section of $K_{X'/S'}$ over $U\cap X'$. We remark that as a relative canonical form $\sigma$ can be written as follows:
Assume that on some open subset of $U'$ the functions $(z^1,\ldots,z^n)$ define local coordinates for $U_s$, $s\in S'$. Then
 \begin{equation}\label{E:sigmainz}
 \sigma =\left(\frac{\pt(f^1,\ldots,f^d)}{\pt(z^{n+1},\ldots,z^{n+d})}\right)^{-1}{\hspace{-5mm}}\cdot dz^1\we\ldots\we dz^n.
 \end{equation}
In particular, $K_{U_s}$ is trivial. There exist $n$ bounded holomorphic functions on $U_s$ which form a local coordinate near $x$ because $U_s$ is embedded in $U$.
Therefore, for any $a\in\CC$, Ohsawa-Takegoshi extension theorem (e.g.~Proposition~\ref{P:OT}) says that there exists $\phi\in\mathcal O(U_s)$ such that $\phi(x)=a$ and
\begin{equation*}
\int_{U_s}\abs{\phi}^2e^{-m\theta_s}(\rho_s)^n
\le C'_s\abs{a}^2e^{-m\theta_s(x)},
\end{equation*}
where the constant $C'_s$ does not depend on $m$.
%It is remarkable to note that the above extension theorem from a point like Corollary~13.9 in \cite{dem} for a Stein ambient space is easily obtained from Theorem~13.6 in \cite{dem}.
If we choose $a\in\CC$ such that the right hand side is $1$, then
\begin{equation*}
\int_{U_s}\abs{\phi}^2e^{-m\theta_s}(\rho_s)^n \le 1, \text{ i.e. } \phi \in \mathcal H^{(m)}_s.
\end{equation*}
It follows that
\begin{equation*}
  \frac{1}{m} \log|\phi(x)|^2=\frac{1}{m} \log|a|^2 \geq \theta_s(x) - \frac{1}{m}\log C'_s ,
\end{equation*}
hence
\begin{equation*}
\vp_m^s(x)
\ge
\theta_s(x)+\frac{1}{m}\log C'_s.
\end{equation*}
This completes the proof.
\end{proof}

\begin{proposition}\label{P:OT}
Let $(W,\omega)$ be a Stein manifold with a K\"ahler metric such that $K_W$ is trivial.
Suppose that there exists bounded holomorphic functions $q^1,\ldots,q^n$ which form a local coordinate near a given point $x\in W$. Then for $a\in \CC$, there exist a global holomorphic function $\phi\in \mathcal O(W)$ with the following properties:
  \begin{itemize}
    \item[(i)] $\phi(x)= a$
    \item[(ii)]
    for any \psh\ function $\Psi$ on $W$
    $$
    \int_W |\phi|^2 e^{-\Psi}dV_{W,\omega}\leq C_n\frac{1}{|\Lambda^n(dq)(x)|_\omega^2} \cdot |a|^2 e^{-\Psi(x)},
    $$
    where the constant $C_n$ depends only upon $n$, and not on $\Psi$.
  \end{itemize}
\end{proposition}

\begin{proof}
Let $C:=e\sup_W\sqrt{\sum_{j=1}^n\abs{q^j}^2}$.
We consider $q=\frac{1}{C}(q^1,\ldots,q^n)$ as a holomorphic section of the trivial vector bundle $E = W\times\CC^n$. 
We denote by $\abs{q}$ the Euclidean norm of the section $s$.
Then we have
$$
\abs{q}^2 \leq \frac{1}{e^2}
$$
holds.
Note that the curvature $\Theta_E$ of $E$ vanishes.

The (trivial) anti-canonicalline bundle $-K_W$  is equipped with the hermitian metric $h_L= e^{-\Psi_L}$ with $\Psi_L =  \Psi$. Then
$$
\im \Theta_{-K_X} + n \im\ddbar \log \abs{q}^2 \geq 0
$$
in terms of currents.

Now we  apply \cite[Theorem 13.6]{dem}: There exists a global holomorphic section $\phi$ of $K_X-K_X=\mathcal O(W)$ with $\phi(x)=a$ and
$$
\int_W \frac{|\phi|^2  e^{-\Psi}}{|q|^{2n}(-\log |q|)^2}dV_{W,\omega}
\leq
C_n\frac{|a|^2}{|\Lambda^n(dq)(x)|_\omega^2} e^{ - \Psi(x)}.
$$
Since $\abs{q}\leq (1/e)$, it follows that
$$
\int_W |\phi|^2 e^{-\Psi}dV_{W,\omega}  \leq
C_n\frac{|a|^2}{|\Lambda^n(dq)(x)|_\omega^2} e^{ - \Psi(x)}.
$$
\end{proof}

\begin{proof}[Proof of Theorem \ref{T:main_theorem}]

Let $\phi\in\mathcal H_s^{(m)}$ be a holomorphic function on $U_s$ with $\norm{\phi}_{m,s}^2\le1$. Then the H\"older inequality implies that
\begin{equation}
\begin{aligned}\label{E:upper_bound}
\int_{U_s}\abs{\phi}^{2/m}e^{-\theta_s}\rho^n_s
&\le
\paren{\int_{U_s}\abs{\phi}^2e^{-m\theta_s}\rho^n_s
}^{\frac{1}{m}}
\paren{\int_{U_s}\rho^n_s}^{\frac{m-1}{m}}\\
&\le
\paren{\vol(X_s)}^{\frac{m-1}{m}}
\le C
\end{aligned}
\end{equation}
where $C$ can be taken as $\mathrm{max(1,\vol(X_s))}$, and the volume of $X_s$ is given by \eqref{E:norm} being taken with respect to the \ka\ class  $\omega\vert_{X_s}$. Hence this constant is independent of $m$ and $s$.

On the other hand, the equation \eqref{E:potential} implies that in local coordinates, it can be written as
\begin{equation*}
e^{-\theta_s}\rho^n_s
=
e^{-\psi_U}
e^{-\eta_s}\rho^n_s
=
e^{-\psi_U}\omega_s^n
=
\frac{dV_X}{\omega^n\wedge f^*dV_S}
\cdot\omega_s^n
=
\frac{dV_X}{f^*dV_S},
\end{equation*}
where $dV_X$ and $dV_S$ are Euclidean volume form in $U$ and $p(U)$ with respect to the coordinates $(z^1,\ldots,z^{n+d})$ and $(s^1,\ldots,s^d)$, respectively.
Combining with \eqref{E:upper_bound}, it follows that
\begin{equation*}
\int_{U_s}\abs{\phi}^{2/m}\frac{dV_X}{f^*dV_S}
\le C.
\end{equation*}

Applying the $L^{2/m}$ version of the Ohsawa-Takegoshi extension theorem in \cite{Berndtsson:Paun2} to
$$
\phi
\paren{
	\frac{dz^1\we\dots\we dz^{n+d}}{f^*\paren{ds^1\we\dots\we ds^d}}
}^{\otimes m},
$$
there exists a $m$-pluricanonical form $\hat F=F\paren{dz^1\we\dots\we dz^{n+d}}^{\otimes m}$ in $U$ with the following properties:
\begin{itemize}
\item[(1)] The restriction of $F$ to $U_s$ is equal to $\phi$.
\item[(2)] There exists a numerical constant $C_0>0$ independent of $m$ and $s$ such that
	\begin{equation*}
	\int_U\abs{F}^{2/m}dV_X
	\le
	C_0
	\int_{U_s}\abs{\phi}^{2/m} \frac{dV_X}{f^*(dV_S)}.
	\end{equation*}
\end{itemize}
It is essential that the constant $C_0$ does not depend on $s$. In fact, the Ohsawa-Takegoshi theorem states that the constant $C_0$ only depends on the sup norm of the holomorphic functions that define the smooth subvariety, interpreted as section of a certain vector bundle. In our case the smooth subvariety is $U_s$ and the section is given by the holomorphic family. Hence the sup norms of the sections are uniformly bounded in terms of the $s$-variable.

For invoking the Ohsawa-Takegoshi extension theorem, however, the holomorphic function $\phi$ is assumed to be integrable with respect to the Ohsawa measure. (For the details, see \cite{Berndtsson:Paun2, dem, Paun})
\medskip

Now the mean value inequality can be applied to the global holomorphic function $F$ on $U$:
The mean value theorem implies that there exists a constant which depends only on $\omega$ and $r$ such that
\begin{equation*}
\abs{F(x)}^{2/m}
\le
C_r\int_{B_r(x)}\abs{F(z)}^{2/m}dV_X.
\end{equation*}
All together, it follows that $\abs{F(x)}^{2/m}$ is bounded from above by a constant which is independent of $m$ and $s$. Now we have an estimate for the restriction $F|{U_s}$. Therefore the weight function $\theta_s$ has the same property (by Theorem~\ref{T:demailly} above).

This argument is carried out for all $\theta_s$, and hence the function $\theta$ is uniformly bounded on $V\cap X'$ where $V\subset\subset U$.
This implies that the curvature $\Theta_{h^\rho_{X'/S'}}(K_{X'/S'})$ extends as a $d$-closed positive real $(1,1)$-current on the total space $X$.
It completes the proof.
\end{proof}
The proof of Corollary~\ref{C:wp} was given in Section~\ref{S:int}.

\section{Application to the projective case}\label{S:proj}

\subsection{Relative Riemann-Roch Theorem for hermitian vector bundles over \ka\ manifolds}
This part is based upon the theorem of Bismut, Gilet and Soul\'e, which generalizes the Riemann-Roch-Hirzebruch-Grothendieck theorem from cohomology classes to distinguished representatives. Given a proper, smooth \ka\ morphism $f:Y\to S$ with relative \ka\ form $\omega_{Y/S}$ and a hermitian vector bundle $(F,h)$ on $Y$, there exists a Quillen metric $h^Q$ on the determinant line bundle
$$
\lambda=\lambda(F):= {\rm det}f_!(F)
$$
taken in the derived category satisfying

\begin{theorem}[\cite{bgs}]\label{T:bgs}
The Chern form of the determinant line bundle $\lambda(F)$ on the base $S$ is equal to the component in degree two of the following fiber integral.
   \begin{equation}\label{E:bgs}
      c_1(\lambda(F),h^Q)= -\left[\int_{Y/S}\textit{td} (Y/S,\omega_{Y/S})\textit{ch}(F,h)\right]^{(2)}
   \end{equation}
 Here $\textit{ch}$ and $\textit{td}$ resp.\ stand for the Chern and Todd character forms resp.
\end{theorem}
Based upon the universal properties of the construction, it was generalized from hermitian vector bundles $(F,h^F)$ to elements of the Grothendieck group, i.e.\ formal differences of such objects. It proved to be consistent to define $ch(G-H,h^G,h^H)= ch(G,h^G)- ch(H,h^H)$.

\subsection{Application to degenerating families of \cy\ manifolds}

We consider the situation of Theorem~\ref{T:main_theorem}, and assume that there exists a hermitian, holomorphic line bundle $(L,h_0)$ on $X$ such that $\omega= c_1(L,h_0)$.

Following \cite[Section 10]{f-s} we define
$$
F=(K_{X'/S'}-K^{-1}_{X'/S'})\otimes(L-L^{-1})^{\otimes n}
$$
on $X$.
Note that the virtual bundles $K_{X'/S'}-K^{-1}_{X'/S'}$ and $L-L^{-1}$ are of rank zero so that the Chern character form in question is
$$
c_1(F,h^F)= 2 c_1( K_{X'/S'}-K^{-1}_{X'/S'}, h^\rho ) \we 2^n c_1(L,h^L)^n + \ldots,
$$
where terms of degree larger than $2n+2$ are omitted. At this point we use the normalization \eqref{E:star} for $\rho$.

Now by \eqref{E:bgs} and \eqref{E:wpfib} the curvature of the determinant line bundle is
\begin{proposition}
$$
c_1(\lambda(F),h^Q) =  2^{n+1} \omega^{WP}.
$$
\end{proposition}
We apply \cite[Thm.\ 2]{Schumacher} (cf.\ \cite{sch:ext}), and get the following fact for a map $f:X\to S$ as in this section.
\begin{theorem}
After a blow up $\wt S\to S$ of the set $S\backslash S'$ of singular values of $f$, the determinant line bundle $(\lambda(F),h^Q)$ extends to $\wt S$ as a holomorphic line bundle equipped with a singular hermitian metric, whose curvature current is positive.
\end{theorem}


\begin{thebibliography} {2010}

%\bibitem{Aubin} Aubin, T., \emph{Equations du type Monge-Amp\`ere sur les vari\'et\'es k\"ahleriennes compactes}, C. R. Acad. Sci. Paris S\'er. A-B 283 (1976), no. 3, Aiii, A119--A121.

\bibitem{Berndtsson:Paun1} Berndtsson, B., P\v{a}un, M., \emph{Bergman kernels and the pseudoeffectivity of relative canonical bundles}, Duke Math. J. 145 (2008), no. 2, 341--378

\bibitem{Berndtsson:Paun2}
Berndtsson, B., P\v{a}un, M.,
\emph{Quantitative extensions of pluricanonical forms and closed positive currents}, Nagoya Math. J. 205 (2012), 25-65.

\bibitem{bgs}
Bismut, J.-M., Gillet, H.,  Soul\'e, C., Analytic torsion and holomorphic determinant bundles. I: Bott-Chern forms and analytic torsion. II: Direct images and Bott-Chern forms. III: Quillen metrics on holomorphic determinants. Commun.\ Math.\ Phys.\ \textbf{115} (1988), 49--78, 79--126, 301--351.

% \bibitem{Braun}
% Braun, M.,
% \emph{Positivit\"at relativer kanonischer
% B\"undel und Krummung h\"oherer
% direkter Bildgarben auf Familien von
% Calabi-Yau-Mannigfaltigkeiten},
% Doctoral thesis in Philipps-Universit\"at Marburg.

\bibitem{Braun:Choi:Schumacher}
Braun, M., Choi, Y.-J., Schumacher, G.,
\emph{K\"ahler forms for families of Calabi-Yau manifolds},
arXiv:1702.07886

% \bibitem{Demailly:agbook} Demailly, J.P.: Complex Analytic and Differential Geometry, Grenoble % 1997.

\bibitem{dem}
Demailly, J.-P., \emph{Analytic methods in algebraic geometry, on the web page of the
author, December 2009.}

\bibitem{f-s}
Fujiki, A., Schumacher, G., The moduli space of extremal compact \ka\  manifolds and generalized Weil-Petersson metrics. Publ.\ Res.\ Inst.\ Math.\ Sci.\ {\bf 26}, 101--183 (1990).

% \bibitem{Griffiths} Griffiths, P.A., \emph{Curvature properties of the Hodge bundles (Notes written by Loring Tu) Topics in Transcendental Algebraic Geometry}, Annals of Mathematics Studies. Princeton University Press, Princeton (1984)


% \bibitem{Kodaira} Kodaira, K., \emph{Complex manifolds and deformation of complex structures}, Grundlehren der Mathematischen Wissenschaften, 283. Springer-Verlag, New York, 1986. x+465 pp.

% \bibitem{Kodaira:Spencer} Kodaira, K., Spencer, D. C., \emph{On deformations of complex analytic structures. I, II}, Ann. of Math. (2) 67 1958 328--466.

\bibitem{Paun} P\v{a}un, M.,
\emph{Relative adjoint transcendental classes and Albanese maps of compact K\"ahler manifolds with nef Ricci curvature},
arxiv:1209.2195[math.CV]

\bibitem{Schumacher}
Schumacher, G.,
 \emph{Positivity of relative canonical bundles and applications},
 Invent. Math. 190 (2012), no. 1, 1--56.

\bibitem{sch:ext}
Schumacher, G., 
\emph{An extension theorem for Hermitian line bundles. Analytic and Algebraic Geometry}, 
225--237, Hindustan Book Agency, New Delhi, 2017, arXiv:1507.06195.


\bibitem{Yau}
Yau, S.-T.,
\emph{On the Ricci curvature of a compact K\"ahler manifold and the complex Monge-Amp\`ere equation. I},
Comm. Pure Appl. Math., no. 31(3), 339--411, 1978.

\end{thebibliography}
\end{document}